\documentclass[11pt,leqno]{amsart}

\usepackage{amsfonts,amssymb,amsmath,amscd}

\usepackage[margin=1in]{geometry}
\usepackage[all,cmtip]{xy}
\usepackage[bookmarks]{hyperref}

\newtheorem{thm}{Theorem}[section]
\newtheorem{lem}[thm]{Lemma}

\newtheorem{prop}[thm]{Proposition}

\theoremstyle{definition}

\newtheorem{rem}[thm]{Remark}

\numberwithin{equation}{thm}

\newcommand{\C}{\mathbb{C}}
\newcommand{\E}{\mathbb{E}}
\newcommand{\F}{\mathbb{F}}
\newcommand{\N}{\mathbb{N}}

\newcommand{\Z}{\mathbb{Z}}

\newcommand{\sN}{\mathcal{N}}

\newcommand{\sV}{\mathcal{V}}

\newcommand{\fb}{\mathfrak{b}}

\newcommand{\g}{\mathfrak{g}}

\newcommand{\fh}{\mathfrak{h}}

\newcommand{\fp}{\mathfrak{p}}

\newcommand{\fu}{\mathfrak{u}}

\newcommand{\ch}{\operatorname{ch}}

\newcommand{\End}{\operatorname{End}}

\newcommand{\Ext}{\operatorname{Ext}}

\newcommand{\opH}{\operatorname{H}}

\newcommand{\heig}{\operatorname{ht}}

\newcommand{\id}{\operatorname{Id}}

\newcommand{\res}{\operatorname{res}}

\newcommand{\wheig}{\operatorname{wht}}

\newcommand{\Uz}{U_\zeta}
\newcommand{\Uzg}{\Uz(\g)}

\newcommand{\uz}{u_\zeta}
\newcommand{\uzg}{\uz(\g)}
\newcommand{\uzb}{\uz(\fb)}
\newcommand{\ug}{u(\g)}

\newcommand{\CZ}{\overline{C}^-_{\Z}}

\newcommand{\subgrp}[1]{\langle #1 \rangle}
\newcommand{\set}[1]{\left\{ #1 \right\}}
\newcommand{\abs}[1]{\left| #1 \right|}

\begin{document}
\title[Differentiating the Weyl generic dimension formula and support varieties]{\large{\bf Differentiating the Weyl generic dimension formula and support varieties for quantum groups}}

\author{Christopher M.\ Drupieski}
\address{
Department of Mathematics \\
University of Virginia \\
Charlottesville, VA~22903}
\email{cmd6a@virginia.edu}
\thanks{The first author was partially supported as a graduate student by NSF grant DMS-0701116}

\author{\ Daniel K.\ Nakano}
\address{
Department of Mathematics \\
University of Georgia \\
Athens, GA~30602}
\email{nakano@math.uga.edu}
\thanks{Research of the second author was partially supported by NSF grant DMS-0654169}

\author{\ Brian J.\ Parshall}
\address{
Department of Mathematics \\
University of Virginia \\
Charlottesville, VA~22903}
\email{bjp8w@virginia.edu}
\thanks{Research of the third author was partially supported by NSF grant DMS-0701116}

\subjclass[2000]{Primary 17B55, 20G; Secondary 17B50}

\begin{abstract} The authors compute the support varieties of all irreducible modules for the small quantum group $\uzg$, where $\g$ is a simple complex Lie algebra, and $\zeta$ is a primitive $\ell$-th root of unity with $\ell$ larger than the Coxeter number of $\g$. The calculation employs the prior calculations and techniques of Ostrik and of Nakano--Parshall--Vella, as well as deep results involving the validity of the Lusztig character formula for quantum groups and the positivity of parabolic Kazhdan-Lusztig polynomials for the affine Weyl group. Analogous support variety calculations are provided for the first Frobenius kernel $G_{1}$ of a reductive algebraic group scheme $G$ defined over the prime field $\F_p$.
\end{abstract}

\maketitle

\section{Introduction} \label{section:introduction}

Let $\uzg$ be the small quantum group associated to the simple finite-dimensional complex Lie algebra $\g$ with parameter $\zeta \in \C$ a primitive $\ell$-th root of unity. Assume that $\ell > h$, the Coxeter number of $\g$. Ginzburg and Kumar \cite{GK} proved that the cohomology algebra $R:=\opH^{2\bullet}(\uzg,\C)$ is isomorphic to the coordinate algebra $\C[\sN]$ of the closed, affine subvariety $\sN=\sN(\g)$ of $\g$ consisting of nilpotent elements. Arkhipov, Bezurkavnikov and Ginzburg \cite{ABG} used this calculation as a starting point to present a new proof of Lusztig's character formula for quantum groups when $\ell > h$. The \cite{ABG} approach demonstrated that understanding character theory for irreducible modules at the representation theoretic level is closely related to the cohomology at the derived level, where the geometry becomes more transparent.

Given a finite-dimensional $\uzg$-module $M$, the annihilator in $R$ of the $\Ext$-group $\Ext_{\uzg}^\bullet(M,M)$ defines a closed subvariety $\sV_{\uzg}(M)$ of $\sN$, called the support variety of $M$. Support varieties provide a method for understanding the interplay between the underlying geometry and the overall representation theory. For the small quantum group, the support varieties for the restriction to $\uzg$ of Weyl modules for the Lusztig quantum group $U_\zeta(\g)$ have been calculated by Ostrik \cite{Ost}, and in subsequent work by Bendel, Nakano, Parshall and Pillen \cite{BNPP}. Bezrukavnikov \cite{Be} recently calculated certain support spaces related to the support varieties of the quantized tilting modules for $\Uzg$, obtaining results similar to those conjectured by Humphreys \cite{Hum2} for the support varieties of tilting modules for algebraic groups.

In this paper we provide an explicit determination of the support varieties $\sV_{\uzg}(L)$ when $L$ is an arbitrary irreducible $\uzg$-module. In
addition to using deep theoretic results on support varieties established in \cite{Ost,NPV}, our calculation makes essential use of the Lusztig character formula for irreducible (type 1, integrable) $U_\zeta(\g)$-modules \cite{KL}, the positivity of parabolic Kazhdan-Lusztig polynomials for Weyl groups associated to symmetrizable Kac--Moody Lie algebras \cite{KT}, and certain divisibility properties of generic Weyl dimension (Laurent) polynomials. These divisibility properties, which explain part of the title of this paper, are presented in Section \ref{section:difgendim}. The main result is proved in Section \ref{section:quantumsupport}.

In Section \ref{section:poschar} we establish that the calculation of support varieties for irreducible modules holds when $\uzg$ is replaced by the restricted enveloping algebra $u(\g)$ (or, equivalently, the first Frobenius kernel $G_1$) associated to a simple, simply-connected algebraic group $G$ defined over an algebraically closed field $k$ of positive characteristic $p$. The calculation assumes that $p\geq h$, \emph{and that the modular Lusztig character formula holds for $G$ for all restricted dominant weights.} The latter result is known to be true for $p$ sufficiently large, depending on the root system $\Phi$ of $\g$; see \cite{AJS} and the subsequent work of Fiebig \cite{F}, which provides explicit (and large) bounds. Some instances for which the Lusztig character formula is known to be true are listed in Section \ref{section:poschar}. At present, it is still expected that Lusztig's original bound of $p \geq h$ will hold true.

Some evidence for the calculations in this paper of support varieties for irreducible modules exists already in the literature. The cases in which the high weight of $L$ is regular or lies on a single wall (i.e., the subregular case) were established in \cite{PW}. In fact, the results in \cite{PW} on the generic dimension motivate the results of Section \ref{section:difgendim}. For irreducible $G$-modules having \emph{regular} high weights, the calculation given in Section \ref{section:poschar} has already been shown in \cite{NPV} and was attributed there to Jantzen. It has been known for some time (cf.\ \cite{CPS}) that the validity of the Lusztig character formula completely determines, through parity considerations, the groups $\Ext_A^\bullet(L,L)$, for $A \in \{G_1,\uzg\}$ and $L,L'$ irreducible $A$-modules with regular high weights. In fact, the dimensions of these cohomology groups are given in terms of Kazhdan-Lusztig polynomials. Conversely, essentially no results are known for these $\Ext$-groups when $L,L'$ have singular high weights. Thus, from this viewpoint, it seems remarkable that the support varieties for \emph{all} irreducible modules can be determined explicitly in the quantum case or in the modular case assuming the validity of the Lusztig character formula.

\bigskip

\begin{center} {\bf Some preliminary notation and conventions} \end{center}

All the following notation is standard.

\begin{enumerate}
\item $\g$: finite-dimensional, simple complex Lie algebra.

\item $\fh$: Cartan subalgebra or a maximal toral subalgebra.

\item $\Phi$: corresponding (irreducible) root system.

\item $\Pi=\{ \alpha_1,\ldots,\alpha_n \},\Phi^+,\alpha_0$: , simple roots, positive roots, and maximal short root.

\item $\fb$: Borel subalgebra $\fb$ of $\g$ consisting of the span of negative root vectors and $\fh$, and opposite to $\fb^{+}$ (span of positive root vectors and $\fh$).

\item $\E$: Euclidean space spanned by $\Phi$.

\item $Q = \Z\Phi$: root lattice in $\E$

\item $\subgrp{\cdot,\cdot}$: inner product on $\E$, normalized so that $\subgrp{\alpha,\alpha^\vee}=2$ if $\alpha \in \Phi$ is any short root.

\item $\rho$: Weyl weight defined by $\rho=\frac{1}{2}\sum_{\alpha\in\Phi^+}\alpha$.

\item $h= \subgrp{\rho,\alpha_0^\vee}+1$: Coxeter number of $\Phi$.

\item $\alpha^\vee=2\alpha/\subgrp{\alpha,\alpha}$: coroot of $\alpha\in \Phi$.

\item $X=\Z \varpi_1 \oplus \cdots \oplus \Z \varpi_n$: weight lattice in $\E$, where the fundamental dominant weights $\varpi_i$ are defined by $\langle\varpi_i,\alpha_j^\vee\rangle=\delta_{ij}$, $1\leq i,j\leq n$.

\item $X^+=\N\varpi_1+\cdots+\N\varpi_n$: cone of dominant weights.

\item $X^+_\ell=\{\lambda\in X^+: \subgrp{\lambda,\alpha^\vee} < \ell,\; \forall \alpha\in\Pi\}$: the set of $\ell$-restricted dominant weights.

\item $s_\beta: \E \rightarrow \E$ ($\beta \in \Phi$): orthogonal reflection in the hyperplane $H_\beta$ of vectors orthogonal to $\beta$.

\item $W \subset {\mathbb O}({\mathbb E})$: Weyl group of $\Phi$, generated by the orthogonal reflections $\set{s_{\alpha_1},\ldots,s_{\alpha_n}}$.

\item $W_\ell = Q \rtimes W$: affine Weyl group, generated by the affine reflections $s_{\alpha,r}:{\mathbb E}\to{\mathbb E}$ defined by $s_{\alpha,r}(x)=x-[\langle x,\alpha^\vee\rangle-rl]\alpha$, $\alpha\in\Phi$, $r\in\mathbb Z$. For $\theta\in Q$, let $t_{\ell\theta}:{\mathbb E}\to{\mathbb E}$ be the translation operator in $W_\ell$ given by $x\mapsto x+\ell\theta$. The affine Weyl group $W_\ell$ is a Coxeter group with fundamental system $S_\ell=\{s_{\alpha_1},\cdots, s_{\alpha_n}\}\cup\{s_{\alpha_0,-1}\}$.

\item $l:W_\ell\to\mathbb N$: usual length function on $W_\ell$.

\item $l|_W$: length function on the parabolic subgroup $W$ of $W_\ell$.

\item $\Phi_{\lambda,\ell} =\{\alpha\in\Phi : \subgrp{\lambda+\rho,\alpha^\vee} \equiv 0 \mod \ell \}$ for $\lambda \in X$. When $\ell$ is clear from context, denote $\Phi_{\lambda,\ell}$ simply by $\Phi_\lambda$. Set $\Phi_\lambda^+ = \Phi^+ \cap \Phi_\lambda$.

\item $C^-=\{\lambda\in {\mathbb E} : -\ell< \langle\lambda+\rho,\alpha^\vee\rangle< 0,\,\forall \alpha\in\Phi^+\}$. The closure $\overline{C^-}$ is a fundamental domain for the ``dot" action $\cdot$ of $W_\ell$ on $\mathbb E$ (which is defined by $w \cdot \lambda= w(\lambda+\rho)-\rho$). Let $C^-_\Z=C^-\cap X$ and $\CZ=\overline{C^-} \cap X$.

\item $\chi(\lambda)=\sum_{w\in W}(-1)^{l(w)}e(w(\lambda+\rho))/\sum_{w\in W}(-1)^{l(w)}e(w\rho)\in\Z[X]$ for $\lambda\in X$: This is Weyl's character formula if $\lambda\in X^+$.

\item $\Psi_\ell(t)\in\Z[t]$: cyclotomic polynomial for a primitive $\ell$-th root of unity $\zeta \in \C$.
\end{enumerate}

\bigskip

Throughout this paper we will assume that {\it $\ell$ is an odd positive integer, $l> h$, and $(l,r)=1$ for all bad primes $r$ of $\Phi$}. In Section \ref{section:poschar} we also assume that $\ell = p$ is a prime integer. The assumption that $(l,r)=1$ whenever $r$ is bad for $\Phi$ guarantees that $\Phi_\lambda$ is a closed subroot system of $\Phi$, and that there exists $w\in W$ and a subset $J\subseteq \Pi$ such that $\Phi_\lambda= w(\Phi_J)$, where $\Phi_J=\Z J \cap \Phi$.

\section{Differentiating the generic dimension} \label{section:difgendim}

For $\theta\in Q$, write $\theta=\sum_{i=1}^nm_i\alpha_i$ ($m_i \in \Z$). The \emph{height} of $\theta$ is defined by $\heig(\theta)=\sum_{i=1}^n m_i$. A {\it weighted} height on $X$ will be defined and used later. We require the following elementary result.

\begin{lem}\label{firstlemma}
For $\beta\in\Phi^+$, $l(s_\beta)< 2\heig(\beta)$.
\end{lem}

\begin{proof}
The result is true if $\heig(\beta)=1$, i.e., if $\beta\in\Pi$. So assume that $\heig(\beta)>1$, and that the result is true for positive roots of smaller height. Choose $\alpha \in \Pi$ so that $s_\alpha(\beta)=\gamma\in\Phi^+$ with $\heig(\gamma)<\heig(\beta)$. Then $s_\beta=s_\alpha s_\gamma s_\alpha$. Since $\langle\gamma,\alpha\rangle<0$, it is easily verified that $l(s_\beta)=l(s_\gamma)+2$. Then
\[
l(s_\beta)=l(s_\gamma)+2<2\heig(\gamma)+2\leq 2\heig(\beta),
\]
as required.
\end{proof}

Fix $\lambda^-\in \CZ$ throughout this section. An element $w \in W_\ell$ is called dominant for $\lambda^-$ provided that $w \cdot \lambda^-\in X^+$. Write $w=t_{\ell\theta} x$ with $\theta \in Q$, $x \in W$. Since $x(\lambda^-+\rho)+\ell \theta \in \rho + X^+$, and
\[
\abs{ \subgrp{ x(\lambda^- + \rho),\alpha^\vee } } = \abs{ \subgrp{ \lambda^-+\rho,x^{-1}\alpha^\vee } } \leq \ell, \, \forall \alpha \in \Phi,
\]
it follows that $\theta\in X^+$. In addition, $w$ is called \emph{minimal} dominant for $\lambda^-$ if it has minimal length among all $y\in W_\ell$ such that $y \cdot\lambda^-= w \cdot \lambda^-$.

Let $w=t_{\ell\theta}x\in W_\ell$ with $\theta\in Q$, $x \in W$. If $\theta\in X^+$, then it follows that
\begin{equation} \label{lengthformula}
\ell(w)=\ell(x) + 2\heig(\theta).
\end{equation}
This result is proved in \cite[Prop. 1.23]{IM}. A routine adjustment must be made in the formula given there, since a different set of fundamental reflections for $W_\ell$ is used.

\begin{lem} \label{secondlem}
Let $w=t_{\ell\theta}x$ be minimal dominant for $\lambda^-\in\CZ$. Given $\alpha\in\Phi_{\lambda^-}^+$, $x\alpha\in -\Phi^+$ if and only if $\langle\lambda^-+\rho,\alpha^\vee\rangle=-\ell$.
\end{lem}

\begin{proof}
Suppose that $\alpha \in \Phi_{\lambda^-}^+$. Then $\langle\lambda^-+\rho,\alpha^\vee\rangle\in\{0,-\ell\}$. If this value is $0$, then $xs_\alpha \cdot \lambda^-=x \cdot \lambda^-$. By hypothesis on $w$, this means that
\[ 2\heig(\theta)+ l(x) = l(w) < l(ws_\alpha)=2\heig(\theta) +l(xs_\alpha), \]
so that $l(xs_\alpha)>l(x)$, and hence $x\alpha>0$. Thus, if $x\alpha<0$, then necessarily $\langle\lambda^-+\rho,\alpha^\vee\rangle=-\ell$. Conversely, assume that $\langle\lambda^-+\rho,\alpha^\vee\rangle=-\ell$. We will show that $x\alpha<0$. Suppose otherwise, viz., $x\alpha>0$. We have $s_{\alpha,-1}\cdot\lambda^-=\lambda^-$. Also,
\[ ws_{\alpha,-1}=t_{\ell\theta}xt_{-\ell\alpha}s_\alpha = t_{\ell\theta-\ell x\alpha}xs_\alpha. \]
Since $ws_{\alpha,-1} \cdot \lambda^- = w \cdot \lambda^-$ is dominant, $\theta - x\alpha$ must be dominant. Then, by \eqref{lengthformula}, $ws_{\alpha,-1}$ has length equal to $l(xs_\alpha) + 2\heig(\theta-x\alpha)$. But
\begin{align*}
l(ws_{\alpha,-1}) &= l(xs_\alpha) + 2\heig(\theta -x\alpha) \\
&= l(s_{x\alpha}x)+2\heig(\theta)-2\heig(x\alpha) \\
&\leq l(x) + l(s_{x\alpha}) + 2\heig(\theta) -2\heig(x\alpha) <l(w),
\end{align*}
since Lemma \ref{firstlemma} guarantees that $l(s_{x\alpha})< 2\heig(x\alpha)$ when $x\alpha\in\Phi^+$. This inequality contradicts the minimality of $w$, so we must conclude that if $\subgrp{\lambda^-+\rho,\alpha^\vee} = -\ell$, then $x\alpha < 0$.
\end{proof}

Following \cite{PW}, we will use a weighted height function $\wheig: X \to \Z[\frac{1}{2}]$. For $\alpha \in \Phi$, let $d_\alpha=\subgrp{\alpha,\alpha}/2 =\subgrp{\alpha,\alpha}/\subgrp{\alpha_0,\alpha_0} \in \set{1,2,3}$. Given $\lambda = \sum_{\alpha \in \Pi} r_\alpha \alpha \in X$ ($r_\alpha\in\mathbb Q$), put
\begin{equation} \label{weightedheight}
\wheig(\lambda):=\sum_{\alpha\in\Pi}r_\alpha d_\alpha= \frac{2\subgrp{\lambda,\rho}}{\subgrp{\alpha_0,\alpha_0}} = \frac{1}{2} \sum_{\alpha \in \Phi^+} d_\alpha \langle \lambda,\alpha^\vee \rangle.
\end{equation}
See \cite[Lemma 1.1]{PW} for the verification that these quantities are all equal. Given a finite-dimensional $X$-graded vector space $V=\bigoplus_{\lambda\in X}V_\lambda$, its \emph{generic dimension} is the Laurent polynomial
\begin{equation} \label{genericdimensionformula}
\dim_t V := \sum_{\lambda \in X} (\dim V_\lambda) t^{-2\wheig(\lambda)}\in\Z[t,t^{-1}]
\end{equation}
We also put $\ch(V)=\sum_{\lambda\in X} (\dim V_\lambda) e(\lambda)$ for the character of $V$.

For $\lambda\in X$, set
\begin{equation} \label{theDs}
D_\lambda(t)=\prod_{\alpha \in \Phi^+} (t^{d_\alpha \subgrp{\lambda+\rho,\alpha^\vee}}-t^{-d_\alpha \subgrp{\lambda+\rho,\alpha^\vee}}) \in \Z[t,t^{-1}].
\end{equation}

\begin{lem} \label{lem:Dquotient} \cite[Theorem 1.3]{PW}
Suppose that $V$ is a finite-dimensional $X$-graded vector space such that $\ch(V)=\chi(\lambda)$ for some $\lambda\in X^+$.
Then
\begin{equation} \label{eq:genericdimension}
\dim_t V=D_\lambda(t)/D_0(t).
\end{equation}
\end{lem}

We call \eqref{eq:genericdimension} the Weyl generic dimension formula. Its value at $t=1$ gives Weyl's classical dimension formula for the irreducible $\g$-module of high weight $\lambda$.

Let $\lambda^-\in \CZ$ as before, and let $w=t_{\ell \theta}x$, $\theta\in X^+\cap Q$, $x\in W$, be minimal dominant for $\lambda^-$. Set $\lambda = w \cdot \lambda^-$, and set $s=|\Phi^+_\lambda|$. For $\alpha\in\Phi^+$, $2d_\alpha \langle \lambda + \rho,\alpha^\vee \rangle$ is divisible by $\ell$ if and only if $\alpha \in \Phi_\lambda^+$. Also, by our assumptions, $\ell$ does not divide any of the $2d_\alpha\langle\rho,\alpha^\vee\rangle$. It follows that the cyclotomic polynomial $\Psi_\ell(t)$ occurs as a factor of $t^{d_\alpha\langle\lambda+\rho,\alpha^\vee\rangle}-t^{-d_\alpha\langle\lambda+\rho,\alpha^\vee\rangle}$ in $\Z[t,t^{-1}]$ if and only if $\alpha\in\Phi^+_\lambda$, hence that $\Psi_\ell(t)$ occurs as a factor of $D_\lambda(t)$ in $\Z[t,t^{-1}]$ precisely $s$ times. In particular, $\Psi_\ell(t)$ is relatively prime to $D_0(t)$.

For the main result of this paper, we need to calculate the value at $t=\zeta$ of the Laurent polynomial
\begin{equation*}
D^{(s)}_\lambda(t):=\frac{d^s}{dt^s}D_\lambda(t),
\end{equation*}
obtained by differentiating $D_\lambda(t)$ $s$ times.

\begin{thm} \label{maincalculation}
Fix $\lambda^-\in \CZ$. Let $w=t_{\ell\theta}x\in W_\ell$ be minimal dominant for $\lambda^-$, and put $\lambda=w\cdot\lambda^-$. Set $s = \abs{\Phi_\lambda^+}$, and set
\[ a_{\lambda^-}=\abs{ \set{ \alpha\in\Phi^+_{\lambda^-} : \subgrp{ \lambda^-+\rho,\alpha^\vee}=-\ell } }. \]
Then
\begin{multline*}
0 \neq D^{(s)}_\lambda(\zeta) = \\
(-1)^{l(w) - (a_{\lambda^-})} (s!)
\left(\prod_{\alpha \in \Phi^+_\lambda} 2d_\alpha \subgrp{\lambda+\rho,\alpha^\vee} \zeta^{-1} \right)
\left(\prod_{\alpha \in \Phi^+ \backslash \Phi^+_{\lambda^-}} \zeta^{d_\alpha \subgrp{\lambda^-+\rho,\alpha^\vee}}-\zeta^{-d_\alpha \subgrp{\lambda^-+\rho, \alpha^\vee}} \right).
\end{multline*}
\end{thm}

\begin{proof}
Write $f_\alpha(t)= t^{d_\alpha\langle\lambda+\rho,\alpha^\vee\rangle}-t^{-d_\alpha\langle\lambda+\rho,\alpha^\vee\rangle}$, so that
$D_\lambda(t)=\prod_{\alpha\in\Phi^+} f_\alpha(t)$. If $(d^i/dt^i) f_\alpha(t)$ is denoted by $f_\alpha^{(i)}(t)$, then $D^{(s)}_\lambda(t)$ is a sum of terms
\begin{equation} \label{derivative}
\left[ s!/(\prod i_\alpha !) \right] \cdot \prod f_\alpha^{(i_\alpha)}(t)
\end{equation}
over distinct sequences $(i_\alpha)_{\alpha\in\Phi^+}$ of non-negative integers $i_\alpha$ summing to $s$. Since $\Psi_\ell(t)$ divides $f_\alpha(t)$ precisely when $\alpha \in \Phi_\lambda^+$ (and then divides it with multiplicity one), the only terms in \eqref{derivative} that do not vanish upon the substitution $t=\zeta$ are those in which $i_\alpha=1$ for all $\alpha\in\Phi^+_\lambda$ (and thus $i_\alpha=0$ for all $\alpha\in \Phi^+\backslash\Phi^+_\lambda$). However,
\[
f_\alpha^{(1)}(t) = f'_\alpha(t) = d_\alpha \langle\lambda+\rho,\alpha^\vee\rangle (t^{d_\alpha \langle \lambda+\rho,\alpha^\vee \rangle-1}+t^{-d_\alpha\langle \lambda+\rho,\alpha^\vee \rangle-1}),
\]
so that, for $\alpha\in \Phi^+_\lambda$,
\[ f'_\alpha(\zeta)=2d_\alpha\langle\lambda+\rho,\alpha^\vee\rangle \zeta^{-1}. \]
Furthermore, each $\langle\lambda+\rho,\alpha^\vee\rangle$ is a positive integer because $\lambda\in X^+$.

Next, for $\alpha \in \Phi^+ \backslash \Phi^+_\lambda$,
\[ \zeta^{d_\alpha\langle\lambda+\rho,\alpha^\vee\rangle}=\zeta^{d_\alpha\langle x(\lambda^-+\rho),\alpha^\vee\rangle} \zeta^{d_\alpha \ell\langle\theta,\alpha^\vee \rangle}=\zeta^{d_\alpha\langle\lambda^-+\rho,x^{-1}\alpha^\vee\rangle}. \]
Observe that $\alpha \notin \Phi_\lambda$ implies that $x^{-1}\alpha \notin \Phi_{\lambda^-}$. If $x^{-1}\alpha<0$, then we write the (non-zero) quantity
\[
\zeta^{d_\alpha\langle\lambda^-+\rho,x^{-1}\alpha^\vee\rangle}-\zeta^{-d_\alpha\langle\lambda^-+\rho, x^{-1}\alpha^\vee\rangle}
\]
as
\[ - \left( \zeta^{d_\alpha\langle\lambda^-+\rho,-x^{-1}\alpha^\vee\rangle}-\zeta^{-d_\alpha\langle\lambda^-+\rho, -x^{-1}\alpha^\vee\rangle} \right). \]
By Lemma \ref{secondlem}, there are $l(w)-(a_{\lambda^-})$ such sign changes. The theorem now follows.
\end{proof}

In the next section, it will be convenient to write
\begin{equation} \label{basicformula}
E_{\lambda^-}(\zeta) = \prod_{\alpha\in\Phi^+\backslash\Phi^+_{\lambda^-}}(\zeta^{d_\alpha\langle\lambda^-+\rho,\alpha^\vee\rangle}-\zeta^{-d_\alpha\langle\lambda^-+\rho, \alpha^\vee\rangle}).
\end{equation}

\section{Support varieties for quantum irreducible modules}\label{section:quantumsupport}

For $\lambda\in X^+$, let $L_\zeta(\lambda)$ be the irreducible, type 1 integrable $U_\zeta(\g)$-module of highest weight $\lambda$. Similarly, let $\Delta_\zeta(\lambda)$ and $\nabla_\zeta(\lambda)$ denote the standard and costandard (i.e., Weyl and induced) modules of highest weight $\lambda$. Then $\Delta_\zeta(\lambda)$ (resp.\ $\nabla_\zeta(\lambda)$) has head (resp.\ socle) $L_\zeta(\lambda)$, with all other composition factors $L_\zeta(\mu)$ satisfying $\mu<\lambda$ in the usual partial ordering on $X$. Furthermore, if $L_\zeta(\mu)$ is a composition factor of $\Delta_\zeta(\lambda)$ (resp.\ $\nabla_\zeta(\lambda)$), then $\mu$ is linked to $\lambda$ (i.e., $\mu$ is conjugate to $\lambda$ under the dot action of $W_\ell$), and hence $\Phi_\lambda$ and $\Phi_\mu$ are $W$-conjugate.

Recall that $W_\ell$ is generated as a group by the fundamental system $S_\ell \subset W_\ell$. Given $I \subseteq S_\ell$, set $W_{\ell,I} = \subgrp{I} \leq W_\ell$, and set $W_\ell^I = \set{w \in W_\ell : l(w) \leq l(ws) \, \forall \, s \in W_{\ell,I}}$. Let $\leqslant$ denote the Chevalley--Bruhat partial ordering on $W_\ell$. Given $y \leqslant w$ in $W_\ell$, $P_{y,w}(q)$ is the Kazhdan--Lusztig polynomial associated to the pair $(y,w)$. In \cite{Deo}, Deodhar introduced two generalizations of the $P_{y,w}$'s, called parabolic Kazhdan--Lusztig polynomials, which depend on a choice of subset $I \subseteq S_\ell$, and a choice of a root $u$ of the equation $u^2=q+(q-1)u$, i.e., $u=-1$ or $u=q$. Given $I \subseteq S_\ell$, and given $(y,w) \in W_\ell^I \times W_\ell^I$ with $y \leqslant w$, the parabolic Kazhdan--Lusztig polynomial $P_{y,w}^{I,-1}$ associated to the root $u=q$ is related to the usual Kazhdan--Lusztig polynomials by the following equation \cite[Remark 3.8]{Deo}:
\begin{equation} \label{parabolicKLpolynomial}
P_{y,w}^{I,-1} = \sum_{x \in W_I,yx \leqslant w} (-1)^{l(x)} P_{yx,w}.
\end{equation}
We are following the notational convention of \cite{KT}, so the superscript in $P_{y,w}^{I,a}$ indicates the opposite root
of the equation $u^2=q+(q-1)u$; see \cite[Remark 2.1]{KT}. If $y \nleqslant w$, then $P_{y,w}^{I,-1} = 0$. By \cite[Corollary 4.1]{KT}, the coefficients of the $P_{y,w}^{I,-1}$ are non-negative integers. In fact, the coefficients are interpreted there as multiplicities of composition factors in Hodge modules associated to Schubert varieties (an alternative approach using affine Hecke algebras is provided in \cite{GH}).

Fix $\lambda^-\in \CZ$. The stabilizer in $W_\ell$ of $\lambda^-$ is defined by $W_{\ell,\lambda^-} = \set{ w \in W_\ell \,|\, w \cdot \lambda^- = \lambda^-}$; it is generated as a group by the set $I:= W_{\ell,\lambda^-} \cap S_\ell$ \cite[II.6.3]{Jan}. Then $W_{\ell,\lambda^-} = W_{\ell,I} := \subgrp{I} \leq W_\ell$ is a parabolic subgroup of $W_\ell$. If $w \in W_{\ell}$ is minimal dominant for $\lambda^-$, then $w \in W_\ell^I$.

\begin{prop} \label{KLformula}
Let $w \in W_\ell$ be minimal dominant for $\lambda^-$, and write $\lambda = w \cdot \lambda^-$. Let $I \subseteq S_\ell$ be such that $W_{\ell,\lambda^-} = W_{\ell,I}$. Then
\begin{equation}
\ch L_\zeta(\lambda) = \sum_{y \in W_\ell^I} (-1)^{l(w)-l(y)} P_{y,w}^{I,-1}(1) \ch \Delta_\zeta(y \cdot \lambda^-).
\end{equation}
\end{prop}

\begin{proof}
According to \cite[Theorem 6.4]{T},
\begin{equation} \label{charformula}
\ch L_\zeta(\lambda) = \sum_{y \in W_\ell,y \leqslant w, y \cdot \lambda^- \in X^+} (-1)^{l(w)-l(y)} P_{y,w}(1) \ch \Delta_\zeta(y \cdot \lambda^-).
\end{equation}
If $y \in W_\ell$ is not dominant for $\lambda^-$, then $\ch \Delta_\zeta(y \cdot \lambda^-) = 0$. Also, if $y \nleqslant w$, then $P_{y,w} = 0$. Then \eqref{charformula} can be rewritten as
\begin{align*}
\ch L_\zeta(\lambda) &= \sum_{y \in W_\ell^I} (-1)^{l(w)-l(y)} \left( \sum_{x \in W_{\ell,I}} (-1)^{l(x)} P_{yx,w}(1) \right) \ch \Delta_\zeta(y \cdot \lambda^-) \\
&= \sum_{y \in W_\ell^I} (-1)^{l(w)-l(y)} P_{y,w}^{I,-1}(1) \ch \Delta_\zeta(y \cdot \lambda^-)
\end{align*}
by \eqref{parabolicKLpolynomial}.
\end{proof}

Now choose $J \subseteq \Pi$ such that $\Phi_{\lambda^-}$ is $W$-conjugate to $\Phi_J$. Let
\begin{equation} \label{eq:nilpotentradical}
{\mathfrak u}_J=\sum_{\alpha\in\Phi^+\backslash\Phi_J^+}\g_{-\alpha}\subset\g
\end{equation}
be the nilpotent radical of the (negative) standard parabolic subalgebra $\fp_J \supseteq \fb$ determined by $J$. Let $G$ be the simple complex algebraic group with Lie algebra $\g$. Recall that $G\cdot{\mathfrak u}_J$ is a closed, irreducible subvariety of the nullcone $\sN(\g)$ of $\g$.

The following theorem was first stated in \cite[Theorem 6.1]{Ost}. We also refer the reader to \cite{BNPP}, which also considers the situation when $\ell \leq h$ and when $(l,r) \neq 1$ for $r$ a bad prime of $\Phi$.

\begin{thm} \label{firstsupport}
Let $\lambda \in X^+$, and choose $J\subseteq \Pi$ such that $w(\Phi_{\lambda})=\Phi_{J}$ for some $w \in W$. Then $\sV_{\uzg}(\Delta_\zeta(\lambda)) = \sV_{\uzg}(\nabla_\zeta(\lambda)) = G\cdot{\mathfrak u}_J$.
\end{thm}

We remark that the subset $J \subseteq \Pi$ and the element $w \in W$ in Theorem \ref{firstsupport} may not be unique, but the Johnston--Richardson theorem \cite{JR} guarantees that if $J,K \subseteq \Pi$ are such that $\Phi_J$ is conjugate to $\Phi_K$ under $W$, then $G \cdot \fu_J = G \cdot \fu_K$, so the variety $\sV_{\uzg}(\nabla_\zeta(\lambda))$ is well-defined. Also, the equality $\sV_{\uzg}(\Delta_\zeta(\lambda)) = \sV_{\uzg}(\nabla_\zeta(\lambda))$ may be seen as follows. First, $\Phi_{\lambda} = w_0(\Phi_{-w_0\lambda})$, so $\sV_{\uzg}(\nabla_\zeta(\lambda)) = \sV_{\uzg}(\nabla_\zeta(-w_0\lambda))$. Next, $\Delta_\zeta(\lambda) = \nabla_\zeta(-w_0\lambda)^*$, so the equality $\sV_{\uzg}(\Delta_\zeta(\lambda)) = \sV_{\uzg}(\nabla_\zeta(-w_0\lambda))$ follows as in \cite[Remark 5.3]{PW} from the fact that the small quantum group $\uzg$ is a quasitriangular Hopf algebra \cite[Example 8.16]{M}.

Now let $M$ be a (type 1, finite-dimensional) $U_\zeta(\g)$-module. Before proving the main theorem, we collect some information concerning the support varieties $\sV_{\uzg}(M)$ and $\sV_{\uzb}(M)$. First, by \cite[Lemma 2.6]{GK}, there exists a rational $B$-algebra isomorphism $\opH^{2\bullet}(\uzb,\C) \cong S^{\bullet}(\fu^*)$, and by \cite[Theorem 3]{GK} there exists a rational $G$-algebra isomorphism $\opH^{2\bullet}(\uzg,\C) \cong \C[\sN]$. Under these identifications, the restriction map $\opH^\bullet(\uzg,\C) \rightarrow \opH^\bullet(\uzb,\C)$ induced by the inclusion $\uzb \subset \uzg$ is simply the restriction of functions from $\sN$ to $\fu$. In particular, the restriction map is surjective.

Now, from the inclusion of algebras $\uzb \subset \uzg$ we get the commutative diagram
\[
\xymatrix{
\opH^\bullet(\uzg,\C) \ar@{->}[r] \ar@{->}[d]^{\res} & \opH^\bullet(\uzg,M \otimes M^*) \ar@{->}[d] \ar@{=}[r] &
\Ext_{\uzg}^\bullet(M,M) \ar@{->}[d] \\
\opH^\bullet(\uzb,\C) \ar@{->}[r] & \opH^\bullet(\uzb,M \otimes M^*) \ar@{=}[r] & \Ext_{\uzb}^\bullet(M,M),
}
\]
where the vertical maps are the obvious restriction maps, and the horizontal maps are induced by the $\uzg$-module homomorphism $\C \rightarrow M \otimes M^* \cong \End_k(M)$, $1 \mapsto \id_M$. From the commutativity of the diagram and the surjectivity of the leftmost restriction homomorphism, we conclude that there exists a closed embedding $\sV_{\uzb}(M) \subseteq \sV_{\uzg}(M) \cap \sN(\fb) = \sV_{\uzg}(M) \cap \fu$.

The support variety $\sV_{\uzg}(M)$ is naturally an algebraic $G$-variety, hence is a union of $G$-orbits. The dimension of $\sV_{\uzg}(M)$ as an algebraic variety is the maximum of the dimensions of the $G$-orbits in $\sV_{\uzg}(M)$. Similarly, $\sV_{\uzb}(M)$ is naturally an algebraic $B$-variety, and its dimension is the maximum of the dimensions of the $B$-orbits in $\sV_{\uzb}(M)$. Since $\sV_{\uzb}(M) \subseteq \sV_{\uzg}(M) \cap \fu$, it follows by a result of Spaltenstein \cite[Proposition 6.7]{Hum} that $\dim \sV_{\uzb}(M) \leq \frac{1}{2} \dim \sV_{\uzg}(M)$.

We are now ready to prove the main theorem.

\begin{thm} \label{maintheorem}
Let $\lambda \in X^+$, and choose $J\subseteq \Pi$ such that $w(\Phi_{\lambda})= \Phi_{J}$ for some $w \in W$. Then
\[ \sV_{\uzg}(L_\zeta(\lambda))=G\cdot{\mathfrak u}_J. \]
\end{thm}

\begin{proof}
We first claim that $\sV_{\uzg}(L_\zeta(\lambda)) \subseteq G\cdot{\mathfrak u}_J$. This is proved in \cite[\S5]{Ost}, but it is easily deduced from the previous theorem: If $\mu$ is linked to $\lambda$ and is minimal among all dominant weights $\leq \lambda$, then $L_\zeta(\mu)=\nabla_\zeta(\mu)$, and $\Phi_\mu$ is $W$-conjugate to $\Phi_\lambda$. Then Theorem \ref{firstsupport} and the Johnston--Richardson theorem \cite{JR} imply that $\sV_{\uzg}(L_\zeta(\mu)) = G \cdot \fu_J = \sV_{\uzg}(\nabla_\zeta(\lambda))$. More generally, if $0 \to M_1 \to M_2 \to M_3 \to 0$ is a short exact sequence of finite-dimensional $\uzg$-modules, then $\sV_{\uzg}(M_{\sigma(1)})\subseteq \sV_{\uzg}(M_{\sigma(2)})\cup \sV_{\uzg}(M_{\sigma(3)})$ for any permutation $\sigma$ of $\set{1,2,3}$ \cite[Lemma 5.2]{PW}. Thus, the full claim follows from an evident induction argument, again using Theorem \ref{firstsupport} together with the remarks at the start of this section.

We next estimate the dimension of $\sV_{u_\zeta(\fb)}(L_\zeta(\lambda))$. We have $\dim \sV_{u_\zeta(\fb)}(L_\zeta(\lambda)) = c_{u_\zeta(\fb)}(L_\zeta(\lambda))$, the complexity of $L_\zeta(\lambda)$ as a $u_\zeta(\fb)$-module. By \cite[Theorem 3.4.1]{NPV}\footnote{Although cast in the situation of algebraic groups, the results of \cite[\S3]{NPV} are clearly applicable in the present
context.}, the complexity $c_{u_\zeta(\fb)}(L_\zeta(\lambda))$ satisfies the inequality $c_{\uzg}(L_\zeta(\lambda))\geq |\Phi^+|-d+1$, where $d$ is any positive integer such that $\Psi_\ell(t)^d$ does not divide the generic dimension $\dim_t L_\zeta(\lambda) \in \Z[t,t^{-1}]$. According to the character formula \eqref{charformula} in Proposition \ref{KLformula} and by Lemma \ref{lem:Dquotient},
\[
\dim_t L_\zeta(\lambda) = \sum_{y \in W_\ell^I} (-1)^{l(w)-l(y)} P_{y,w}^{I,-1}(1) D_{y \cdot \lambda^-}(t)/D_0(t),
\]
where $I \subseteq S_\ell$ is such that $W_{\ell,\lambda^-} = W_{\ell,I}$. Since $D_0(t)$ is relatively prime to $\Psi_\ell(t)$, to determine a lower bound for $c_{u_\zeta(\fb)}(L_\zeta(\lambda))$, it suffices to determine the multiplicity with which $\Psi_\ell(t)$ occurs as a factor in $D_0(t) \cdot \dim_t L_\zeta(\lambda)$. Equivalently, it suffices to determine the multiplicity with which the primitive $\ell$-th root of unity $\zeta \in \C$ occurs as a root of the Laurent polynomial $D_0(t) \cdot \dim_t L_\zeta(\lambda)$.

Set $f(t) = D_0(t) \cdot \dim_t L_\zeta(\lambda)$. If $f^{(i)}(\zeta) = 0$ for all $0 \leq i < n$, but $f^{(n)}(\zeta) \neq 0$, then $\zeta$ occurs as a root of $f$ with multiplicity exactly equal to $n$. Set $s = |\Phi_{\lambda^-}^+|$. Then $s = | \Phi_{y \cdot \lambda^-}^+|$ for any $y \in W_\ell$ by \cite[(3.4.2)]{NPV}. We want to show that $n = s$. Certainly $n \geq s$, because $\Psi_\ell(t)$ occurs as a factor of $D_{y \cdot \lambda^-}(t)$ precisely $s$ times by the discussion following Lemma \ref{lem:Dquotient}. Then, to prove $n=s$, we must show that $f^{(s)}(\zeta) \neq 0$.

Applying Theorem \ref{maincalculation}, we get
\begin{align*}
f^{(s)}(\zeta) &= \sum_{y \in W_\ell^I} (-1)^{l(w)-l(y)}P_{y,w}^{I,-1}(1) D_{y \cdot \lambda^-}^{(s)}(\zeta) \\
&= \sum_{y \in W_\ell^I} (-1)^{l(w)-(a_{\lambda^-})}P_{y,w}^{I,-1}(1) (s!) \left( \prod_{\alpha \in \Phi_{y \cdot \lambda^-}^+} 2d_\alpha \langle y \cdot \lambda^- + \rho, \alpha^\vee \rangle \right) \zeta^{-s} E_{\lambda^-}(\zeta) \\
&= \left( (-1)^{l(w)-(a_{\lambda^-})} (s!) \zeta^{-s} E_{\lambda^-}(\zeta) \right) \cdot \left( \sum_{y \in W_\ell^I} P_{y,w}^{I,-1}(1) \left( \prod_{\alpha \in \Phi_{y \cdot \lambda^-}^+} 2d_\alpha \langle y \cdot \lambda^- + \rho, \alpha^\vee \rangle \right) \right)
\end{align*}
The first term in the product of the last line is non-zero. The second term in the product is a sum of non-negative integers (by the positivity property for the parabolic Kazhdan-Lusztig polynomials). Since $P_{w,w}^{I,-1}(1) = 1$, we conclude that the second term in the product is a strictly positive integer, hence that $f^{(s)}(\zeta) \neq 0$.

Now $\sV_{u_\zeta({\mathfrak b})}(L_\zeta(\lambda))$ has dimension at least $|\Phi^+|-s = \abs{\Phi^+} - \abs{\Phi_{\lambda^-}^+}$. By the discussion preceding the theorem, we have $\dim \sV_{\uzg}(L_\zeta(\lambda)) \geq \abs{\Phi} - \abs{\Phi_{\lambda^-}} = \dim G \cdot {\mathfrak u}_J$. This completes the proof.
\end{proof}

\section{Results in positive characteristic} \label{section:poschar}

In this section, $G$ is a simple, simply-connected algebraic group defined over an algebraically closed field $k$ of positive characteristic $p$. (We leave to the reader the routine task of extending these results to reductive groups.) Fix a maximal torus $T \subset G$, and let $\Phi$ be the root system of $T$ acting on the Lie algebra $\g$. Most of the previous notation, with $\ell$ set equal to $p$, carries over to $G$ with only small changes (i.e., $B \supset T$ is a Borel subgroup whose opposite $B^+$ defines the set $\Phi^+$ of positive roots, etc.). The Lie algebra $\g$ carries a restricted structure; let $\ug$ denote its restricted enveloping algebra. We assume that $p>h$, so that the cohomology algebra $\opH^\bullet(u(\g),k)$ is isomorphic to $k[\sN]$, the coordinate ring of the variety $\sN$ of nilpotent elements in $\mathfrak g$. The result below concerns the support varieties $\sV_{u(\g)}(L(\lambda))$ of the irreducible $G$-modules $L(\lambda)$, $\lambda\in X^+$. If $\lambda=\lambda_0+p\lambda_1$ with $\lambda_0\in X^+_p$ (the restricted weights), then $\sV_{u(\g)}(L(\lambda))=\sV_{u(\g)}(L(\lambda_0))$. Therefore, in computing support varieties for irreducible modules, it suffices to consider only those having restricted highest weights.

Let $\lambda=w\cdot\lambda^-\in X^+$, $\lambda^-\in \CZ$, $w\in W_p$. Assume that $w$ is minimal dominant for $\lambda^-$. The Lusztig character formula asserts
\begin{equation}\label{modularLusztig}
\ch\,L(\lambda)=\sum_{y\in W_p, y\leq w, y\cdot\lambda^-\in X^+}(-1)^{l(w)-l(y)}P_{y,w}(1)\ch\,\Delta(y\cdot\lambda^-),
\end{equation}
where $\Delta(y\cdot\lambda^-)$ is the Weyl module for $G$ of highest weight $y\cdot\lambda^-$. As mentioned in the introduction, \eqref{modularLusztig} holds for all restricted dominant weights $\lambda$, provided that the prime $p$ is sufficiently large (the lower bound on $p$ depending on the root system).\footnote{The Lusztig character formula is also known to hold for restricted weights in the following low rank cases (assuming $p\geq h$): (1) type $A_1$, $p\geq 2=h$; (2) type $A_2$, $p\geq 3=h$; (3) type $B_2$, $p>4=h$; (4) type $G_2$, $p> 9=2h-3$; (5) type $A_3$, $p>4=h$; (6) type $A_4$, $p=5$ or $p=7$. Case (6) for $p=5$ is due independently to L.~Scott (working with undergraduates) and to A.\ Buch and N.\ Lauritzen. The case $p=7$ is due to L.\ Scott (again working with undergraduates). Both these cases required extensive computer application. For more details and references, see \cite{S}.}

Let $J\subseteq\Pi$ such that $\Phi_\lambda$ is $W$-conjugate to $\Phi_J$. By \cite[Proposition 7.4.1]{NPV}, $\sV_{u(\g)}(L(\lambda))\subseteq G\cdot{\mathfrak u}_J$, where ${\mathfrak u}_J$ is defined as in \eqref{eq:nilpotentradical}. With this fact, the proof of the following result is exactly analogous to that of Theorem \ref{maintheorem} (replacing $\Psi_\ell(t)$ by $\Psi_p(t)$, etc.).

\begin{thm} \label{charptheorem}
Assume that $G$ is a simple, simply-connected algebraic group over an algebraically closed field $k$ of characteristic $p>h$. Assume that the
Lusztig character formula \eqref{modularLusztig} holds for all restricted dominant weights. Then, for $\lambda\in X^+$ and $J\subseteq \Pi$ with $w(\Phi_{\lambda})=\Phi_{J}$,
\[ \sV_{u(\g)}(L(\lambda))=G\cdot{\mathfrak u}_J. \]
\end{thm}

\begin{rem} (a)
Suppose $p=h$. It may no longer hold that $A:= \opH^{2\bullet}(u(\g),k)\cong k[{\sN}]$. Even so, it has been proved that the algebraic variety defined by the affine algebra $A$ is {\it homeomorphic} to $\sN$ \cite{SFB1}, \cite{SFB2}. In this case, we identify $\sV_{u(\g)}(L(\lambda))$ with its image in $\mathcal N$, and Theorem \ref{charptheorem} holds with the condition ``$p>h$" replaced by the condition ``$p\geq h$".

(b) For a restricted Lie algebra $\mathfrak g$ (with restriction map $x\mapsto x^{[p]}$ and restricted enveloping algebra $u(\g)$) and a finite dimensional $u(\g)$-module $M$, the support variety $\sV_{u(\g)}(M)$ has an alternate, more concrete description as the set of all $x \in \g$ such that $x^{[p]}=0$ and the induced operator $x_M$ on $M$ has an $r\times r$ Jordan block of size $r<p$. In other words, for $0 \neq x \in \g$ satisfying $x^{[p]}=0$, $x\not\in \sV_{u(\g)}(M)$ if and only if the nilpotent operator $x_M$ acts projectively on $M$ (cf.\ \cite{FP}). At present there is no known concrete realization in $\sN$ for the support varieties of modules over the small quantum group.
\end{rem}


\begin{thebibliography}{AAAA}


\bibitem[\sf ABG]{ABG} S.\ Arkhipov, R.\ Bezrukavnikov, V.\ Ginzburg, Quantum groups, the loop Grassmannian, and the Springer resolution, {\em J.\ Amer.\ Math.\ Soc.} {\sf 17} (2004), 595--678.

\bibitem[\sf AJS]{AJS} H.\ Andersen, J.\ Jantzen, W.\ Soergel, \emph{Representations of quantum groups at a pth root of unity and of semisimple groups in characteristic p}, Ast\'{e}rique {\sf 220} (1994).

\bibitem[\sf BNPP]{BNPP} C.\ Bendel, D.\ Nakano, C.\ Pillen, B.\ Parshall, Cohomology for quantum groups via the geometry of the nullcone, preprint (2007).

\bibitem[\sf Be]{Be} R.\ Bezrukavnikov, Cohomology of tilting modules over quantum groups and $t$-structures on derived categories of coherent sheaves, {\em Invent.\ Math.} {\sf 166} (2006), 327--357.

\bibitem[\sf Bo]{Bo} N.\ Bourbaki, {\it Elements of Mathematics: Lie Groups and Lie Algebras}, Chapters 4--6, Springer (1972).

\bibitem[\sf CPS]{CPS} E.\ Cline, B.\ Parshall, L.\ Scott, Infinitesimal Kazhdan-Lusztig theories, {\em Cont. Math.} {\sf 139} (1992), 43--73.

\bibitem[\sf Deo]{Deo} V.\ Deodhar, On some geometric aspects of Burhat orderings II. The parabolic analogue of Kazhdan--Lusztig polynomials, \emph{J.\ Algebra}, {\sf 111} (1987), 483--506.

\bibitem[\sf F]{F} P.\ Fiebig, An upper bound on the exceptional characteristics for Lusztig's character formula, preprint (2009). arXiv:0811.1674

\bibitem[\sf FP]{FP} E.\ Friedlander, B.\ Parshall, Support varieties for restricted Lie algebras, {\it Invent.\  Math.,} {\sf 86} (1986), 553--562.

\bibitem[\sf GH]{GH} I.\ Grojnowski, M.\ Haiman, Affine Hecke algebras and positivity of LLT and Macdonald polynomials, preprint (2007).

\bibitem[\sf GK]{GK} V.\ Ginzburg, S.\ Kumar, Cohomology of quantum groups at roots of unity, {\em Duke Math.\ Journal}, {\sf 69} (1993), 179--198.

\bibitem[\sf Hum1]{Hum} J.E.\ Humphreys, {\em Conjugacy Classes in Semisimple Algebraic Groups}, Mathematical Surveys and Monographs, vol.\ 43, Amer.\ Math.\ Soc.\ (1995).

\bibitem[\sf Hum2]{Hum2} J.E. Humphreys, Comparing modular representations of semisimple groups and their Lie algebras, {\em Modular Interfaces: Modular Lie algebras, Quantum groups, and Lie Superalgebras, International Press}, Cambridge, MA, 1997.

\bibitem[\sf IM]{IM} N.\ Iwahori, H.\ Matsumoto, On some Bruhat decomposition and the structure of the Hecke rings of $p$-adic Chevalley groups, {\em Inst.\ Hautes \'Etudes Sci.\ Publ Math.} {\sf 25} (1965), 5--48.

\bibitem[\sf Jan]{Jan} J.C.\ Jantzen, \emph{Representations of Algebraic Groups}, second edition, Amer.\ Math.\ Soc., Providence, RI (2003).

\bibitem[\sf JR]{JR} D.S.\ Johnston, R.W.\ Richardson, Conjugacy classes in parabolic subgroups of semisimple algebraic groups. II., \emph{Bull.\ London Math.\ Soc.}, {\sf 9}, (1977), 245--250.

\bibitem[\sf KL]{KL} D.\ Kazhdan, G.\ Lusztig, Tensor structures arising from affine Lie algebras, I--III; IV--VI, \emph{J.\ Amer.\ Math.\ Soc.} {\sf 6} (1993), 905--1011; {\sf 7} (1994), 335--453.

\bibitem[\sf KT]{KT} M.\ Kashiwara, T.\ Tanisaki, Parabolic Kazhdan--Lusztig polynomials and Schubert varieties, \emph{J.\ Algebra}, {\sf 249} (2002), 306--325.

\bibitem [\sf M]{M} E.\ M\"{u}ller, Some topics on Frobenius--Lusztig kernels, II, \emph{J.\ Algebra}, {\sf 206} (1998), 659--681.

\bibitem [\sf NPV]{NPV} D.K.\ Nakano, B.J.\ Parshall, D.C.\ Vella, Support varieties for algebraic groups, {\em J.\ Reine Angew.\ Math.}, {\sf 547} (2002), 15--49.

\bibitem[\sf Ost]{Ost} V.\ Ostrik, Support varieties for quantum groups, {\em Funct.\ Anal.\ and its Appl.}, {\sf 32}, (1998), 237--246.

\bibitem[\sf PW]{PW} B.\ Parshall, J-P.\ Wang, Cohomology of quantum groups: the quantum dimension, {\em Canadian J.\ Math.}, {\sf 45} (1993), 1276--1298.

\bibitem[\sf S]{S} L.\ Scott, Some new examples in $1$-cohomology, {\it J.\ Algebra} {\sf 260} (2003), 416--425.

\bibitem[\sf SFB1]{SFB1} A.\ Suslin, E.M.\ Friedlander, C.P.\ Bendel, Infinitesimal 1-parameter subgroups and cohomology, {\em Jour.\ Amer.\ Math.\ Soc.}, {\sf 10}, (1997), 693-728.

\bibitem[\sf SFB2]{SFB2} A.\ Suslin, E.M.\ Friedlander, C.P.\ Bendel, Support varieties for infinitesimal group schemes, {\em Jour.\ Amer.\ Math.\ Soc.}, {\sf 10}, (1997), 729-759.

\bibitem[\sf T]{T} T.\ Tanisaki, Character formulas of Kazhdan-Lusztig type, {\it Representations of finite dimensional algebras and related topics in Lie theory and geometry}, Fields Inst.\ Commun.\ {\sf 40}, Amer.\ Math.\ Soc., Providence, RI (2004) 261--276.

\end{thebibliography}
\end{document}